\documentclass{ws-ijbc}

\usepackage[utf8]{inputenc}
\usepackage{graphicx}
\usepackage{epstopdf}
\usepackage{float}
\usepackage{amsmath}
\usepackage{mathtools}
\usepackage{color}
\usepackage[caption=false]{subfig}
\usepackage{multicol}
\usepackage{appendix}
\usepackage{tikz}
\usepackage{pgfplots}
\pgfplotsset{compat=1.16}

\usetikzlibrary{decorations.markings}
\usetikzlibrary{calc}
\usetikzlibrary{arrows.meta}
\tikzset{
  mid arrow/.style={postaction={decorate,decoration={
    markings,
    mark=at position 0.5 with {\arrow[scale=1.5]{stealth}}
  }}}
}

\begin{document}

\catchline{}{}{}{}{} 

\markboth{Yu.A. Kuznetsov and J. Hooyman}{Bifurcations of heteroclinic contours in two-parameter planar systems}

\title{Bifurcations of heteroclinic contours in two-parameter planar systems: {O}verview and explicit examples}

\author{Yuri A. Kuznetsov and Joost Hooyman}
\address{Mathematical Institute, Utrecht University,\\
P.O. Box 80010, 3508 TA  Utrecht,\\ The Netherlands}

\maketitle

\begin{history}
\received{(to be inserted by publisher)}
\end{history}

\begin{abstract}
    Smooth planar vector fields containing two hyperbolic saddles may possess contours formed by heteroclinic connections between these saddles. We present an overview of the bifurcations of these contours based on papers by J.W. Reyn and A.V. Dukov.  Additionally, two new explicit polynomial systems containing such contours are derived, which are studied using the bifurcation software {\sc matcont} and are shown to exhibit the theoretically predicted phenomena, including series of heteroclinic connections.
\end{abstract}


\begin{multicols}{2}
\allowdisplaybreaks
\section{Introduction} \label{Sec1}
Consider a vector field in the plane containing two hyperbolic saddles and suppose an orbit leaving one saddle connects to the other saddle and vice versa. This configuration of saddles and orbits is called a {\it heteroclinic contour} (other names used: {\it polygon}, {\it polycycle}). This paper deals with bifurcations occurring in generic two-parameter families of vector fields containing such contours at some critical parameter values.

Planar heteroclinic contours can be subdivided into two categories, namely {\it monodromic} and {\it non-monodromic}. For both cases there already exist detailed theoretical studies. Results on the monodromic case have been first published by  \cite{reyn} and are quite well known. For example, they are included in the surveys  \cite{BazKuzKhib:1989,Sha:1992} and reproduced in the encyclopedic work \cite{ArnoldDSV:1994}.
In contrast, results on the non-monodromic case are much less known, although its detailed study has been first done in \cite{Roi:1989} and  \cite{Sha:1992}, and recently independently repeated by \cite{dukov}.

In both cases, depending on the eigenvalues of the saddles, either one limit cycle or two limit cycles of opposite stability are generated as the result of homoclinic bifurcations. When two limit cycles exist these may collide to form a semi-stable limit cycle at a fold (or saddle-node) of cycles bifurcation.

The non-monodromic case exhibits an additional phenomenon of {\it flashing} (or {\it sparkling}) heteroclinic connections, which may be described as an infinite series of bifurcation curves corresponding to heteroclinic connections from one saddle to the other making arbitrarily many windings around one of the saddles. This phenomenon was already discovered in 1981 by  \cite{maltapalis} but in a less general (one-parameter) context. Two-parameter perturbation of the non-monodromic case provides an example of a situation which is only two-dimensional but nevertheless possesses a bifurcation diagram of considerable complexity, where an infinite number of bifurcation curves emanate from a codim 2 point, see \cite{Roi:1989,Sha:1992,dukov}. Notice, however, that a possibility of such parametric complexity near a non-monodromic contour was also briefly mentioned in \cite{BazKuzKhib:1989}.

The first aim of this paper is to collect the results of all mentioned papers and present them in a unified manner by describing generic bifurcation diagrams in both the monodromic and the non-monodromic cases, including phase portraits (cf. \cite{Sha:1992}).

The second aim of this paper is to derive two polynomial systems which contain heteroclinic connections of the monodromic and non-monodromic kind respectively. These systems may be perturbed to bifurcate as predicted in the theoretical part.
The method to obtain the monodromic example is adap\-ted from a method to derive a polynomial system containing a homoclinic connection by  \cite{sandstede}. The non-monodromic example is derived by modifying a reversible system containing a homoclinic connection. After these explicit polynomial systems are obtained we use the standard numerical software package {\sc matcont} \cite{matcont:2003,matcont:2008,matcontHOM:2012} to continue relevant phase objects in two parameters, thus numerically producing bifurcation diagrams of these systems.

In Appendix A, we summarize for reader's convenience the major facts on flashing heteroclinic bifurcations, assuming that basic facts on codim 1 bifurcations of cycles and connecting orbits of planar vector fields are known (see, e.g. \cite{Ku:2004}). In Appendix \ref{app:melnikov} we show that the splitting of the heteroclinic connections in the monodromic example is regular.
\par\medskip\noindent
{\bf Acknowledgments:}~~The authors would like to thank A.V. Dukov (Moscow State University) for his useful comments on a draft of this paper. Furthermore, the authors are thankful to A.L. Shilnikov (Georgia State University, Atlanta) and L.M. Lerman (State University of Nizhny Novgorod) for informing about early works on bifurcations of non-monodromic heteroclinic contours.

\section{Heteroclinic contours of codim 2} \label{Sec2}
Consider a planar vector field
\begin{equation}
\left\{\begin{array}{rcl}
\dot{x}&=&f(x,y),\\
\dot{y}&=&g(x,y),
\end{array}
\right.
\label{eq:planarsystem}
\end{equation}
where \(f\) and \(g\) are smooth functions from \(\mathbb{R}\times \mathbb{R}\) to \(\mathbb{R}\).
Assume the system (\ref{eq:planarsystem}) has two hyperbolic equilibria of saddle type which we label \(L\) and \(M\). Let $\lambda_s<0<\lambda_u$ and $\mu_s<0<\mu_u$ be the eigenvalues of $L$ and $M$, respectively. Introduce the corresponding \emph{saddle indices} by
$$
\lambda:=-\frac{\lambda_s}{\lambda_u},~~\mu:=-\frac{\mu_s}{\mu_u}.
$$
Also suppose that this system has solutions \(\gamma_1(t)\) satisfying
\[\lim_{t\to-\infty}\gamma_1(t)=L\quad\text{and}\quad\lim_{t\to\infty}\gamma_1(t)=M,\]
and \(\gamma_2(t)\) satisfying
\[\lim_{t\to-\infty}\gamma_1(t)=M\quad\text{and}\quad\lim_{t\to\infty}\gamma_1(t)=L.\]
We see that the orbits of \(\gamma_1(t)\) and \(\gamma_2(t)\) are heteroclinic connections between \(L\) and \(M\) in opposite directions. Together with the saddles they form a heteroclinic contour.
\begin{figure}[H]
  \centering
  \subfloat[Monodromic]{\includegraphics[scale=0.9]{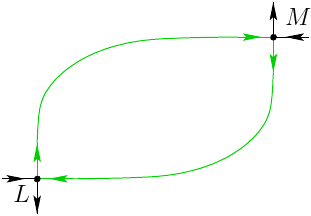}}
  \subfloat[Non-monodromic]{\includegraphics[scale=0.9]{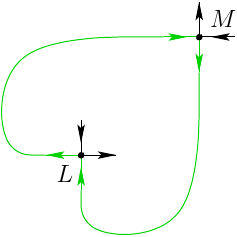}}
  \caption{The two types of heteroclinic contours}
  \label{fig:hetcontours}
\end{figure}
Depending on the position of the stable and unstable manifolds of \(L\) and \(M\) there are two distinct types of heteroclinic contours called \emph{monodromic} and \emph{non-monodromic}, see Fig. \ref{fig:hetcontours}.

To understand this choice of terminology, note that the monodromic contour may be approached as a limit set by solutions inside the contour. It is also possible to contain all stable and unstable manifolds inside the contour, in which case the contour may be approached as a limit set for solutions outside the contour. This case is also monodromic but is topologically equivalent to the depicted one and, thus, will not be considered. In the non-monodromic case, the contour can not be approached as a limit set, neither from the inside nor the outside.

We shall see that the bifurcation diagrams in both cases depend strongly on the values of \(\lambda\), \(\mu\) and also \(\lambda\mu\). Especially, it is important whether these quantities are smaller or larger than 1. Recall that saddles with index greater than one are called \emph{ dissipative}. In principle, the following generic subcases are possible:

\begin{table}[H]
\centering
\begin{tabular}{|l|lll|}
\hline
1 & \(\lambda<1\) & \(\mu>1\) & \(\lambda\mu<1\) \\ \hline
2 & \(\lambda<1\) & \(\mu>1\) & \(\lambda\mu>1\) \\ \hline
3 & \(\lambda>1\) & \(\mu<1\) & \(\lambda\mu>1\) \\ \hline
4 & \(\lambda>1\) & \(\mu<1\) & \(\lambda\mu<1\) \\ \hline
5 & \(\lambda>1\) & \(\mu>1\) & \(\lambda\mu>1\) \\ \hline
6 & \(\lambda<1\) & \(\mu<1\) & \(\lambda\mu<1\) \\ \hline
\end{tabular}
\end{table}

However, reversing time and interchanging the roles of \(L\) and \(M\), we see that there are only two essentially different subcases to consider. More speci\-fi\-cally, reversing time, we see that 1 goes to 3 and 5 to 6. When changing the roles of \(L\) and \(M\), 1 goes to 4 and 2 to 3. Thus if we study only the subcases 1 and 6, the results for all others follow immediately. To conclude, we consider only subcases \(\lambda<1\), \(\mu<1\) and \(\lambda<1\), \(\mu>1\) with \(\lambda\mu<1\) in both of them.

\section{Bifurcations of heteroclinic contours} \label{Sec3}
We are now ready to describe bifurcations of planar monodromic and non-monodromic heteroclinic contours in generic two-parameter systems. Consider a smooth vector field
\begin{equation}
\left\{\begin{array}{rcl}
\dot{x}&=&f(x,y,\alpha),\\
\dot{y}&=&g(x,y,\alpha),
\end{array}
\right.
\label{eq:planarsystempar}
\end{equation}
where \(f\) and \(g\) are smooth functions from \(\mathbb{R} \times \mathbb{R} \times \mathbb{R}^2\) to \(\mathbb{R}\). We assume that at $\alpha=(0,0)$ the system (\ref{eq:planarsystempar}) has a monodromic or non-monodromic heteroclinic contour described in Sec. \ref{Sec2}. Each case is treated separately.

\subsection{Bifurcation diagrams in the monodromic cases}
\label{sec:monodromic}
In a neighborhood of \(L\), we introduce two cross-sections, \(\Sigma_L\) and \(\Pi_L\), on the incoming and outgoing heteroclinic connection respectively, see Fig. \ref{fig:monocontour}. On these cross sections, we introduce coordinates \(\xi_L\) and \(\eta_L\), again respectively. We define these coordinates in such a way that the points \(\xi_L=0\) and \(\eta_L=0\) coincide with the points on the intersection of the stable and unstable manifolds and the corresponding cross-section, and such that positive values of the coordinates correspond with points inside the heteroclinic contour.

\begin{figure}[H]
    \centering
    \includegraphics[scale=1.0]{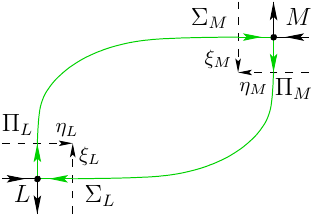}
    \caption{The monodromic contour with all cross-sections and coordinates}
    \label{fig:monocontour}
\end{figure}

Analogously, in a neighborhood \(U_M\) of \(M\), we introduce cross-sections \(\Sigma_M\) and \(\Pi_M\) furnished with coordinates \(\xi_M\) and \(\eta_M\) with analogous properties.

Using the cross-sections, we define the splitting parameters \(\beta=(\beta_1,\beta_2)\) which measure the splitting of the heteroclinic connection along \(\Sigma_L\) and \(\Sigma_M\) respectively.

Summarising the results for both monodromic subcases, we may now sketch their bifurcation diagrams in  Figs. \ref{fig:mono1bifdiag} and \ref{fig:mono2bifdiag}. We denote the curves corresponding to homoclinic connections \(P_L\) and \(P_M\). In the second subcase, the curve in the parameter plane corresponding to the semi-stable (double) limit cycle will be denoted by \(F\). Curves corresponding to heteroclinic connections between the saddles are denoted with the letter \(H\) and are labeled by the saddle from which they leave. Thus the curve corresponding to a heteroclinic connection from \(L\) to \(M\) is labeled with \(H_L\).

\end{multicols}

\begin{figure}[!htb]
  \centering
  \includegraphics[scale=0.55]{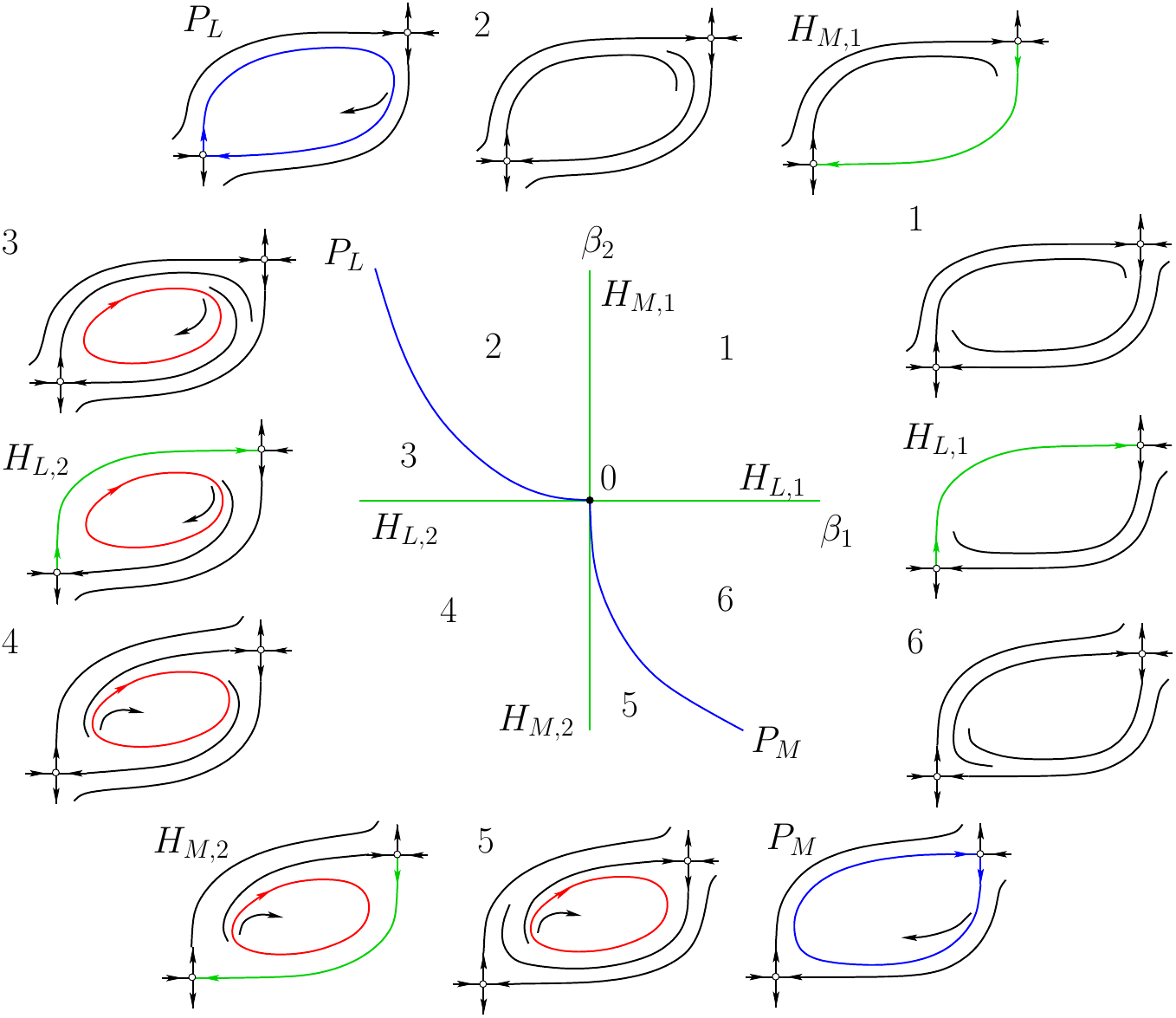}
  \caption{\(\lambda<1\) and \(\mu<1\)}
  \label{fig:mono1bifdiag}
\end{figure}

\begin{figure}[!htb]
  \centering
  \includegraphics[scale=0.55]{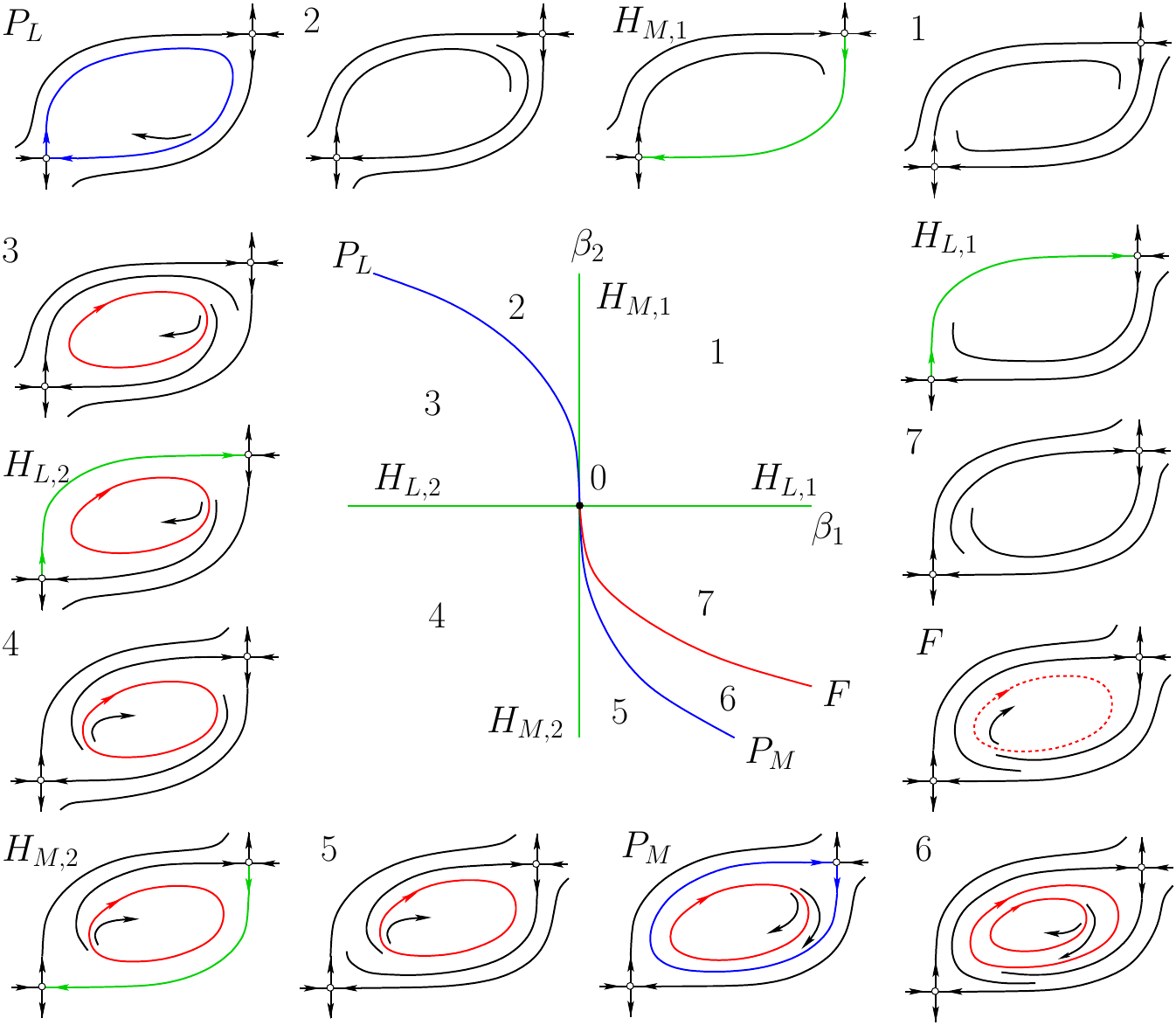}
  \caption{\(\lambda<1\) and \(\mu>1\)}
  \label{fig:mono2bifdiag}
\end{figure}

\begin{multicols}{2}
\allowdisplaybreaks

Choosing initial parameter values on the curve \(P_L\) and varying the parameters through the bifurcation diagram in Fig. \ref{fig:mono1bifdiag} counterclockwise, we observe the following changes of the phase portrait in the first subcase. We begin with a small homoclinic connection of the saddle \(L\). This homoclinic connection is broken outwards and an unstable limit cycle is generated in accordance with the fact that \(\lambda<1\). Then, a heteroclinic connection from \(L\) to \(M\) is formed and destroyed, and thereafter the same happens to a heteroclinic connection from \(M\) to \(L\). Eventually, the unstable limit cycle disappears via a small homoclinic connection of \(M\), which is then broken inwards. This happens in accordance with  \(\lambda<1\). Under further variation of the parameters, heteroclinic connections from \(L\) to \(M\) and back are formed and destroyed before the original homoclinic connection of \(L\) is regained. In short, one might describe these changes as an interplay between two homoclinic bifurcations.

In the second subcase, we perform the same analysis. Starting with parameter values on \(P_L\) and varying them counterclockwise through the bifurcation diagram in Fig. \ref{fig:mono2bifdiag}, we begin with a small homoclinic connection of \(L\) which is destroyed forming an unstable limit cycle. This behaviour is not different from the first subcase, which was to be expected as it still holds that \(\lambda<1\). The evolution does not differ from the first subcase until a small homoclinic connection of \(M\) is formed. In this process, the unstable limit cycle does not disappear but is contained within the homoclinic connection. This happens in accordance with the fact that \(\mu>1\). Furthermore, when this connection is broken inwards another limit cycle, now stable, is generated enclosing the original unstable one.

Eventually, these limit cycles collide, forming a single double limit cycle which is stable from the outside but unstable from the inside. This semi-stable and nonhyperbolic limit cycle is immediately destroyed and further variation of the parameters only creates heteroclinic connections from \(L\) to \(M\) and back, before the original homoclinic connection of \(L\) is again formed. In short, one might describe these changes as the interplay between two homoclinic bifurcations resulting in a fold of cycles bifurcation.

The above results can be obtained by considering a {\it model Poincar\'{e} map} that is constructed as follows. We may define approximate singular maps \(\Delta_L:\Sigma_L\to\Pi_L\) and \(\Delta_M:\Sigma_M\to\Pi_M\) by
\[\Delta_L(\xi_L)=\xi_L^\lambda\quad\text{and}\quad\Delta_M(\xi_M)=\xi_M^\mu\]
near the saddles and truncated regular maps \(Q:\Pi_L\to\Sigma_M\) and \(R:\Pi_M\to\Sigma_L\) by
\[Q(\eta_L)=\beta_2+\theta_1\eta_L\quad\text{and}\quad R(\eta_M)=\beta_1+\theta_2\eta_M,\]
where \(\theta_{1,2}\) are positive constants, near the heteroclinic connections. By composing these maps we arrive at a model Poincaré map \(P:\Sigma_L\to\Sigma_L\), which becomes
\begin{equation}\label{eq:monodromicpoincaremap}
    \begin{split}
        P(\xi_L)&=(R\circ\Delta_M\circ Q\circ\Delta_L)(\xi_L)\\
        &=\beta_1+\theta_2(\beta_2+\theta_1\xi_L^\lambda)^\mu.
    \end{split}
\end{equation}
It has been shown by \cite{reyn} that taking into account higher-order terms in the singular and regular maps in (\ref{eq:monodromicpoincaremap}) does not change topology of the bifurcation diagrams in both subcases.

\graphicspath{{fig/part2/non-monodromic/graphs/case1/}{fig/part2/non-monodromic/graphs/case2/}{fig/part2/non-monodromic/portraits/}}

\subsection{Bifurcation diagrams in the non-monodromic cases}
\label{sec:nonmonodromic}
We now proceed to the non-monodromic case. The treatment will be analogous to that of the monodromic case. Thus, we begin by describing homoclinic connections, then move to limit cycles and their bifurcations, and finally to heteroclinic connections.

The construction of the model maps in the non-monodromic case is virtually the same as in the monodromic case. However, to obtain a singular map describing the behaviour of a solution near the saddle, the cross-sections at \(L\) must be oriented in the opposite direction to allow solutions to pass near the saddles on the positive side of the cross-sections.

The new configuration of the cross sections is pictured in Fig. \ref{fig:nonmonocontour}. Consequently, the splitting parameters at \(L\) are now also measured differently.

The bifurcation diagrams in the two non-monodromic cases are sketched in Figs. \ref{fig:nonmono1bifdiag} and \ref{fig:nonmono2bifdiag}. Once more, we will be able to find two curves, \(P_L\) and \(P_M\), corresponding to homoclinic connections at \(L\) and \(M\) which border a {\it wedge} in the parameter-plane.

\begin{figure}[H]
    \centering
    \includegraphics[scale=1.0]{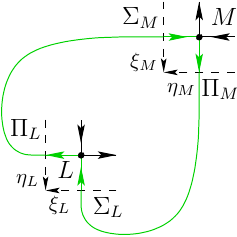}
    \caption{The non-monodromic contour with all cross-sections and coordinates}
    \label{fig:nonmonocontour}
\end{figure}

In the first subcase (\(\lambda<1\) and \(\mu<1\)) we find that outside the wedge no limit cycles are present. If we continuously vary our parameters in this region, we can enter the wedge by crossing its border either through \(P_L\) or \(P_M\). Breaking either one of the associated homoclinic connections generates an unstable limit cycle via a homoclinic bifurcation. This limit cycle exists in the wedge between \(P_L\) and \(P_M\). If we vary the parameters further, a homoclinic connection of the opposing saddle is created and then destroyed, also destroying the unstable limit cycle.

In the second subcase (\(\lambda<1\) and \(\mu>1\)) the wedge between \(P_L\) and \(P_M\) still contains an unstable limit cycle that is generated by breaking the homoclinic connection of \(L\). Varying the parameters further, a homoclinic connection appears for \(M\), now with \(\mu>1\). As a result of this difference, breaking of the this connection does not destroy the unstable limit cycle, but generates an additional stable limit cycle enclosing it. Eventually, these limit cycles collide, momentarily forming a nonhyperbolic semi-stable double limit cycle at a cyclic fold bifurcation at the curve $F$.

We recognize that all these results are similar to the monodromic case. However, the non-monodromic case also presents a new phenomenon: non-trivial heteroclinic connections, i.e. heteroclinic connections which appear when both splitting parameters are non-zero. These connections must wind around \(L\) at least once before reaching \(L\) or \(M\). Actually, in both subcases, there exist two infinite series of curves corresponding to heteroclinic connections between the saddles, with each connection making an increasing number of turns. Thus, we encounter the phenomenon of \emph{flashing heteroclinic connections}, both at the bifurcation of a semi-stable limit cycle as well as at a homoclinic connection (see Appendix A). In the parameter plane, the infinite series of curves corresponding to these non-trivial heteroclinic connections accumulate, in the first subcase at \(P_L\) and \(P_M\), corresponding to the homoclinic connections, and in the second subcase at \(P_L\) and \(F\).

Summarising the results by \cite{dukov} on both non-monodromic subcases, we present the bifurcation diagrams in Figs. \ref{fig:nonmono1bifdiag} and \ref{fig:nonmono2bifdiag}. The curves corresponding to heteroclinic connections are indicated along with the number of turns around the saddle \(L\), shown as a bracketed superscript. For practical reasons, we only supply phase portraits for heteroclinic connections making at most one turn. The accumulating curves corresponding to the connections making more turns are drawn dashed.

\end{multicols}

\begin{figure}[!htb]
  \centering
  \includegraphics[scale=0.7]{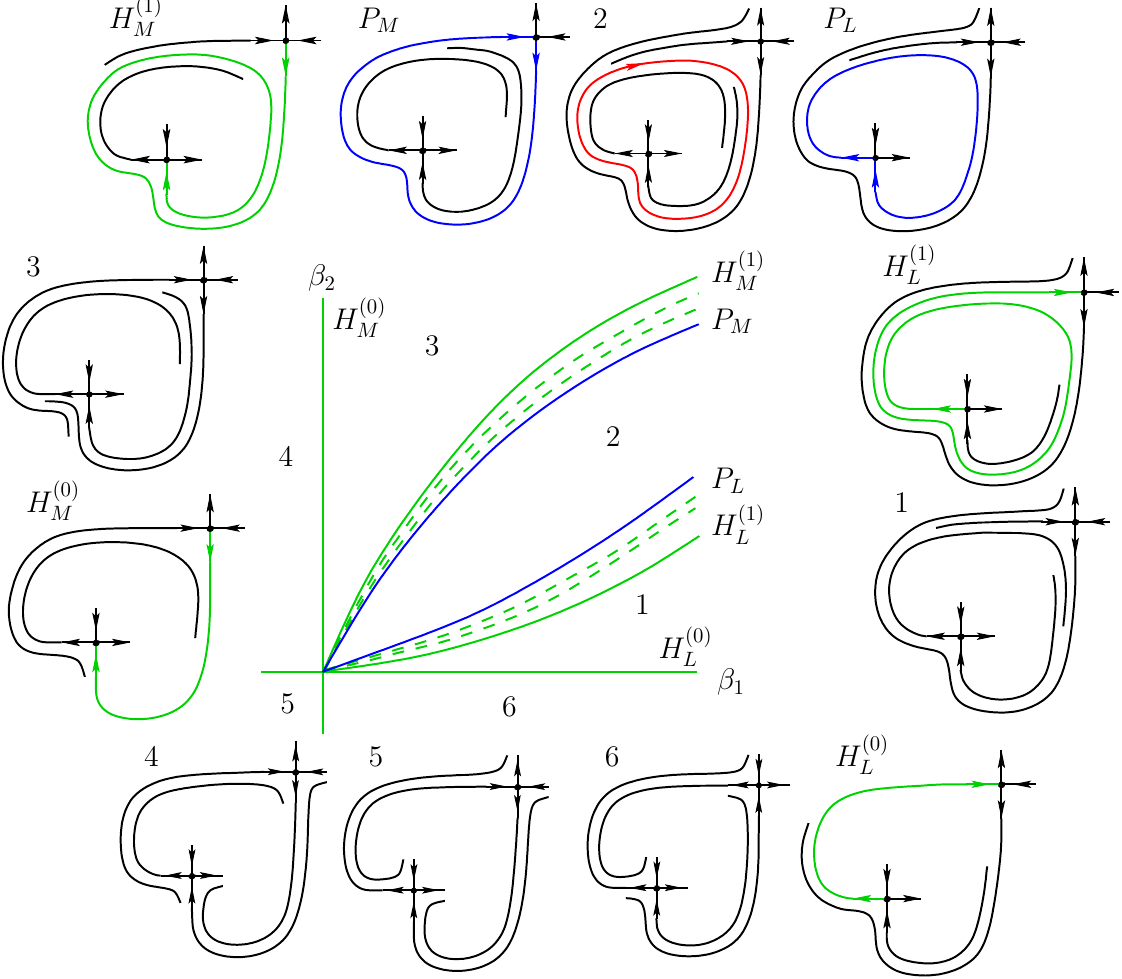}
  \caption{\(\lambda<1\) and \(\mu<1\)}
  \label{fig:nonmono1bifdiag}
\end{figure}

\begin{figure}[!htb]
  \centering
  \includegraphics[scale=0.7]{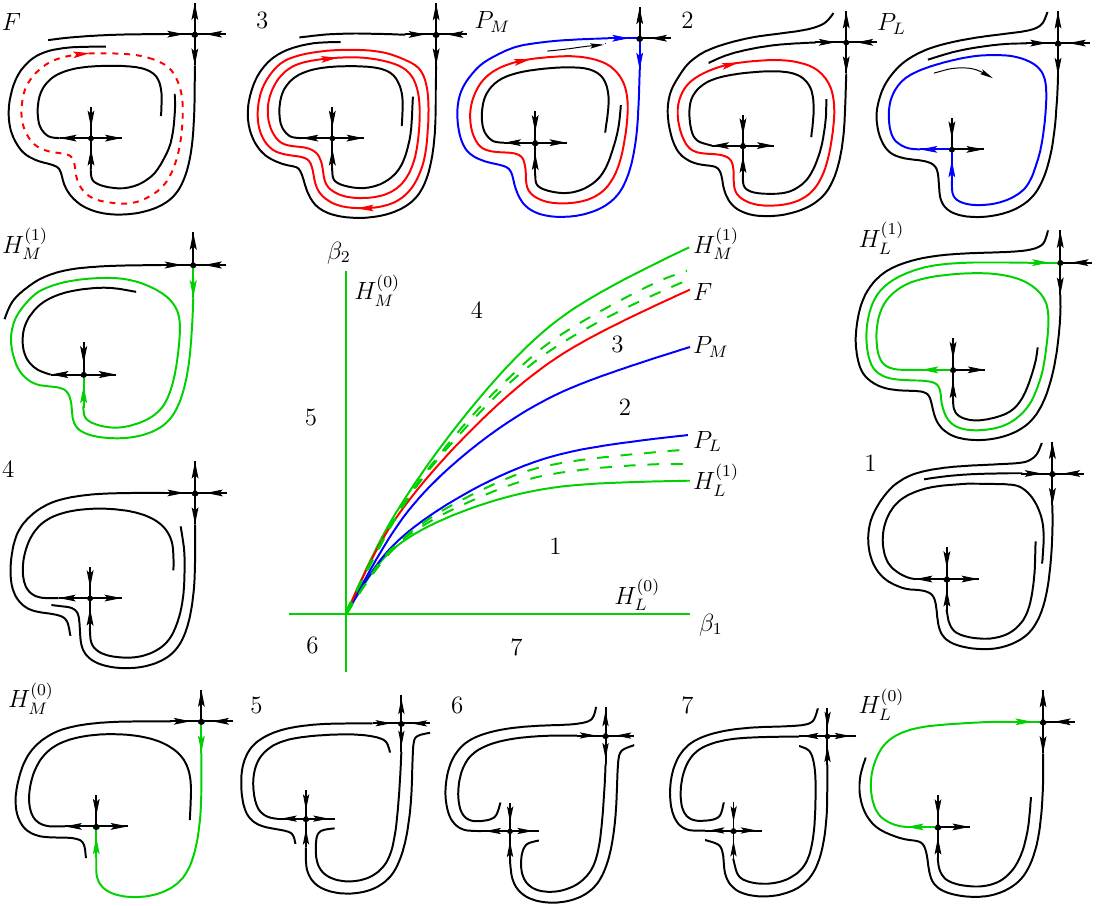}
  \caption{\(\lambda<1\) and \(\mu>1\)}
  \label{fig:nonmono2bifdiag}
\end{figure}

\begin{multicols}{2}

Choosing parameter values on the curve \(H_L^{(0)}\) and varying the parameters through the bifurcation diagram in Fig. \ref{fig:nonmono1bifdiag} counterclockwise, we observe the following changes of the phase portrait. Initially, there is a heteroclinic connection from \(L\) to \(M\) which is destroyed. After variation of the parameters, a second heteroclinic connection is formed which winds around the saddle \(L\) once before reaching \(M\). Varying the parameters further results in an infinite series of flashing heteroclinic connections from \(L\) to \(M\) making an increasing number of turns around \(L\) before eventually a big homoclinic connection of \(L\) is formed. This big homoclinic connection is broken inwards, generating an unstable limit cycle in accordance with the fact that \(\lambda<1\). This limit cycle disappears when a small homoclinic connection of \(M\) is formed, in accordance with \(\mu<1\). Breaking this homoclinic connection results again in an infinite series of flashing heteroclinic connections, now from \(M\) to \(L\) and making a decreasing number of turns. After further variation of the parameters the original heteroclinic connection from \(L\) to \(M\) is regained.

In short, the changes of this situation may be described as an interplay between two homoclinic bifurcations in the presence of two series of flashing heteroclinic connections.

The analysis of the second subcase (see Fig. \ref{fig:nonmono2bifdiag}) yields the same result up to the appearance of a small homoclinic connection of \(M\). In accordance with the fact that \(\mu>1\), the existing unstable limit cycle is not destroyed in this process. As in the monodromic case, the destruction of this small homoclinic orbit generates a stable limit cycle enclosing the stable one. Eventually these two cycles collide, momentarily forming a double limit cycle which is stable from the outside but unstable from the inside. In this case, an infinite series of heteroclinic connections from \(M\) to \(L\) accumulates at this semi-stable nonhyperbolic limit cycle. Thus, when this limit cycle is destroyed we again encounter an infinite series of flashing heteroclinic connections from \(M\) to \(L\) making a decreasing number of turns around \(L\). The further behaviour is the same as in the first subcase. These changes may be described as the interplay between two homoclinic bifurcations yielding a fold of cycles bifurcation in the presence of two series of flashing heteroclinic connections.

All above results can be obtained by considering a {\it model Poincar\'{e} map} that is constructed simi\-lar to the monodromic case. The approximate singular maps \(\Delta_L:\Sigma_L\to\Pi_L\) and \(\Delta_M:\Sigma_M\to\Pi_M\) can be expressed exactly in the same way as before, namely
\[\Delta_L(\xi_L)=\xi_L^\lambda\quad\text{and}\quad\Delta_M(\xi_M)=\xi_M^\mu.\]
On the contrary, the truncated regular maps \(Q:\Pi_L\to\Sigma_M\) and \(R:\Pi_M\to\Sigma_L\) become
\[Q(\eta_L)=\beta_2-\theta_1\eta_\lambda\quad\text{and}\quad R(\eta_M)=\beta_1-\theta_2\eta_M.\]
Thus, in the non-monodromic case we obtain the model Poincaré map \(P:\Sigma_L\to\Sigma_L\) as
\begin{equation}\label{eq:nonmonodromicpoincaremap}
    P(\xi_L)=\beta_1-\theta_2(\beta_2-\theta_1\xi_L^\lambda)^\mu.
\end{equation}
Note that in terms of the model maps we can change between the monodromic and non-monodromic cases through the transformation \((\theta_1,\theta_2)\mapsto(-\theta_1,-\theta_2)\).

It has been shown by \cite{dukov} that taking into account higher-order terms in the singular and regular maps in (\ref{eq:nonmonodromicpoincaremap}) does not change topology of the bifurcation diagrams in both subcases.


\section{Two polynomial examples}
\graphicspath{{fig/part3/}{/}}

Here we present two systems of polynomial ordinary differential equations which illustrate some of the theoretical results described in the previous section. These explicit examples allow for numerical computation of the various phase objects and bifurcation diagrams using {\sc matcont} \cite{matcont:2003,matcont:2008,matcontHOM:2012}, which can be compared to their theoretically predicted counterparts.

Our strategy to find explicit monodromic examples -- originally applied by \cite{sandstede} to construct homoclinic bifurcation examples -- will be to look for an algebraic variety which has the shape of a monodromic contour and then to look for vector fields which are tangent to this variety. For a non-monodromic example a different construction is applied. It is based on perturbing a reversible system containing a big homoclinic connection.

\subsection{A monodromic example}
\subsubsection{Derivation of the example}
To construct a vector field containing a monodromic heteroclinic contour, we will look for an algebraic variety which has a shape similar to such a contour. For this example we choose the union of a parabola with the \(x\)-axis. Specifically, we choose the parabola described by \(y=x(1-x)\), see Fig. \ref{fig:parabolaline}.
\begin{figure}[H]
    \centering
    \begin{tikzpicture}[scale=0.8]
  \begin{axis}[
    axis lines=center,
    xmin=-0.25,
    xmax=1.25,
    ymin=-0.25,
    ymax=0.5,
    xlabel=\(x\),
    ylabel=\(y\),
    ticks=none,
    domain=-0.5:1.5,
    samples=150,
    ]
    \addplot[thick] {0};
    \addplot[thick] {x*(1-x)};
    
    \node[below right] at (axis cs:0,0) {0};
    \node[below left] at (axis cs:1,0) {1};
  \end{axis}
\end{tikzpicture}
    \caption{The union of the parabola \(y=x(1-x)\) with the \(x\)-axis.}
    \label{fig:parabolaline}
\end{figure}
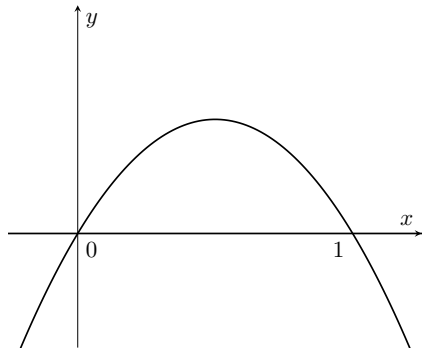
This algebraic variety can be realised as the zero-set of the function
\begin{equation}
    G(x,y)=y(y-x(1-x))
\end{equation}
and we write \(\Gamma=G^{-1}(0)\).

Our next step is to look for two polynomial functions \(f(x,y)\) and \(g(x,y)\) such that the vector field
\begin{equation}\label{eq:vectorfield}
    \begin{cases}
    \dot{x}=f(x,y)\\
    \dot{y}=g(x,y)
    \end{cases}
\end{equation}
is tangent to \(\Gamma\). For this to happen it must hold that
\begin{equation}\label{eq:conditions}
    \langle \nabla G(x,y),
    \begin{pmatrix}
        f(x,y)\\
        g(x,y)
    \end{pmatrix}
    \rangle=0
\end{equation}
for all \((x,y)\in\Gamma\). The equation (\ref{eq:conditions}) will yield a condition from which we can derive preliminary expressions for \(f(x,y)\) and \(g(x,y)\). As we will soon see, it will suffice to use polynomials of second degree. Thus, we set
\begin{equation}
\label{eq:f&g}
    \begin{array}{rcl}
        f(x,y)&=&a_{10}x+a_{01}y+a_{20}x^2+a_{11}xy+a_{02}y^2,\\
        g(x,y)&=&b_{10}x+b_{01}y+b_{20}x^2+b_{11}xy+b_{02}y^2.
    \end{array}
\end{equation}
Then, (\ref{eq:conditions}) together with $G(x,y)=0$ yields the system of two polynomial equations
\begin{equation}
    \label{eq:totalsystem}
        \begin{array}{rcl}
b_{10}x^2 + (a_{10}-2b_{10}+b_{01})xy + (a_{01}-2b_{01})y^2 &-&\\
(b_{10}-b_{20}) x^3 &-&\\
(2a_{10}+2b_{20}-b_{11}+b_{01}-a_{20})x^2y &+&\\
(a_{11}-2b_{11}-2a_{01}+b_{02})xy^2 + (a_{02}-2b_{02})y^3 &-&\\
b_{20}x^4 - (2a_{20}+b_{11})x^3y -  (2a_{11}+b_{02})x^2y^2&-&\\
2a_{02}xy^3 &=&0,\\
        y^2-xy+x^2y&=&0,
    \end{array}
\end{equation}
where the coefficients $a_{jk}$ and $b_{jk}$ are yet unknown.

To study this system we use the software package Maple to calculate its Gröbner basis using the graded reverse lexicographic order for the $x^jy^k$-monomials. This yields a total of four basis polynomials including \(G(x,y)\) itself. To look for solutions of (\ref{eq:totalsystem}), we choose the basis polynomial of the minimal degree and require all its coefficients to be identically zero. This yields the system of linear equations
\begin{equation}\label{eq:linearsystem}
\left\{
    \begin{array}{rcl}
        b_{20}&=&0,\\
        b_{10}-b_{20}&=&0,\\
        b_{10}&=&0,\\
        a_{02}&=&0,\\
        2a_{01}+a_{11}-2a_{20}+b_{11}&=&0,\\
        a_{10}+2b_{10}+a_{20}+2b_{11}&=&0,\\
        a_{02}+2a_{11}-b_{02}&=&0,\\
        2a_{10}+a_{01}-b_{01}+a_{20}+2b_{20}&=&0,
    \end{array}
\right.
\end{equation}
which is solved whenever we set
\begin{eqnarray*}
a_{10}&=&-a_{20},\\
b_{01}&=&a_{01}+a_{10},\\
b_{11}&=&-2a_{01}-2a_{10}-a_{11},\\
b_{02}&=&2a_{11},
\end{eqnarray*}
and all remaining coefficients to zero. With this choice, all basis polynomials except $G$ become identically zero. We see that this leaves us with three free parameters remaining. Introducing
\begin{eqnarray*}
a&:=&a_{10},\\
b&:=&a_{01},\\
c&:=&a_{11},
\end{eqnarray*}
and substituting the result into \(f(x,y)\) and \(g(x,y)\) in (\ref{eq:f&g}), we  now obtain the vector field
\begin{equation}\label{eq:monodromicvectorfield}
    \begin{cases}
        \dot{x}=ax+by+cxy-ax^2,\\
        \dot{y}=(a+b)y-(2a+2b+c)xy+2cy^2,
    \end{cases}
\end{equation}
for which $\Gamma$ is invariant for any choice of \(a\), \(b\) and \(c\). This can now be verified directly, since the left-hand side of (\ref{eq:conditions}) contains $G(x,y)$ as factor.

As expected, this vector field always has equilibria \(L=(0,0)\) and \(M=(1,0)\). However, the types of these equilibria depend on the values of \(a\), \(b\) and \(c\). The eigenvalues at \((0,0)\) are \(a\) and \(a+b\) and the eigenvalues at \((1,0)\) are \(-a\) and \(-a-b-c\).

As a first example, we set \(a=1\), \(b=-3\) and \(c=3/2\). We find that \(\lambda_u=1\) and \(\lambda_s=-2\) so \(\lambda=-\lambda_s/\lambda_u=2>1\), while \(\mu_u=1/2\) and \(\mu_s=-1\) so \(\mu=-\mu_s/\mu_u=2>1\). Obviously, $\lambda\mu>1$ implying that the heteroclinic contour is stable from the inside. We recognize that these saddle indices correspond to the first subcase treated in Section \ref{sec:monodromic}, {\it after the time reversal}. A phase portrait of this system is given in Fig. \ref{fig:phaseportrait2}.

\begin{figure}[H]
    \centering
    \includegraphics[]{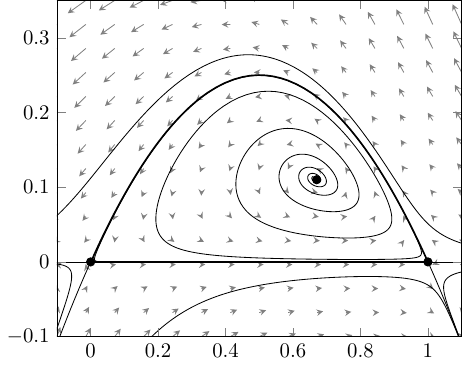}
    \caption{Phase portrait of (\ref{eq:monodromicvectorfield}) for \(a=1\), \(b=-3\) and \(c=3/2\)}
    \label{fig:phaseportrait2}
\end{figure}

\begin{figure}[H]
    \centering
    \includegraphics[]{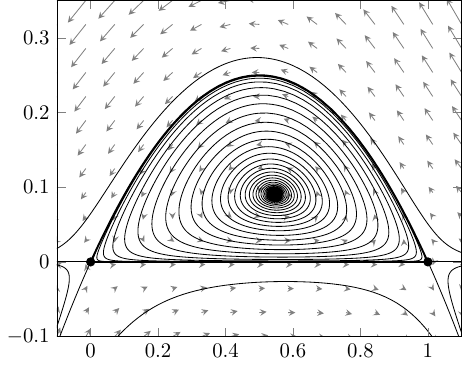}
    \caption{Phase portrait of (\ref{eq:monodromicvectorfield}) for \(a=1\), \(b=-3\) and \(c=1/2\)}
    \label{fig:phaseportrait1}
\end{figure}

As a second example, we set \(a=1\), \(b=-3\) and \(c=1/2\).  Then it remains that \(\lambda=2>1\) but now \(\mu_u=3/2\) and \(\mu_s=-1\) so \(\mu=-\mu_s/\mu_u=2/3<1\) and  \(\lambda\mu=4/3>1\). We recognize that we are in the situation of the second subcase treated in Section \ref{sec:monodromic}. A phase portrait of this system is shown in Fig. \ref{fig:phaseportrait1}.

In order to perform bifurcation analysis we still need to add two parameter-dependent perturbing terms to (\ref{eq:monodromicvectorfield}) which break the heteroclinic connections separately. As these connections can be described by the equations \(y=x(1-x)\) and \(y=0\) this can be achieved by choosing parameters \(\alpha\) and \(\varepsilon\) and adding the perturbing terms \(\alpha(y-x(1-x))\) and \(\varepsilon y\) to (\ref{eq:monodromicvectorfield}). It is clear that the first perturbation does not affect the parabola while the second perturbation does not affect the straight line. We add one perturbing term to \(\dot{x}\) and one to \(\dot{y}\). This yields
\begin{equation}\label{eq:perturbedmonodromicvectorfield}
\left\{
    \begin{array}{rcl}
        \dot{x}&=&ax+by+cxy-ax^2+\varepsilon y,\\
        \dot{y}&=&(a+b)y-(2a+2b+c)xy+2cy^2\\
        &&~~~~~~~~~~~+~\alpha(y-x(1-x)).
    \end{array}
\right.
\end{equation}

\subsubsection{Numerical analysis of the example}
Using the Matlab software package {\sc matcont} \cite{matcont:2003,matcont:2008,matcontHOM:2012}, we perform numerical bifurcation analysis of the system (\ref{eq:perturbedmonodromicvectorfield}) in both subcases. More specifically, we will be able to locate and continue homoclinic connections and cycles which will allow us to compute the bifurcation diagrams of both subcases.

In Appendix \ref{app:melnikov} we show that for the chosen values of the parameters \(a\), \(b\) and \(c\), the splitting of the heteroclinic connections under the given perturbations is regular. This is done by showing that the relevant Melnikov integrals are nonvanishing.

In Section \ref{sec:monodromic}, a general sketch of these diagrams is presented (see Fig. \ref{fig:mono1bifdiag}). The numerical study of these polynomial systems will allow us to compare this general diagram based on the behaviour of model maps to an actual example of such a diagram.

This analysis in the first case yields the bifurcation diagram displayed in Fig. \ref{fig:numericbifurcationdiagramcase1}. This analysis in the second case yields the bifurcation diagram displayed in Fig. \ref{fig:numericbifurcationdiagramcase2}.

To complete our analysis of these systems we compute representative phase portraits for parameter values in all components of the bifurcation diagrams. The contents of these phase portraits were already predicted in the previous part. They are displayed in Figs. \ref{fig:totalhombifdiag1} and \ref{fig:totalhombifdiag2}.

\columnbreak

\begin{figure}[H]
    \centering
    \caption{The parameter plane diagram in the first subcase (\(a=1\), \(b=-3\) and \(c=3/2\))}
    \includegraphics[width=0.4\textwidth]{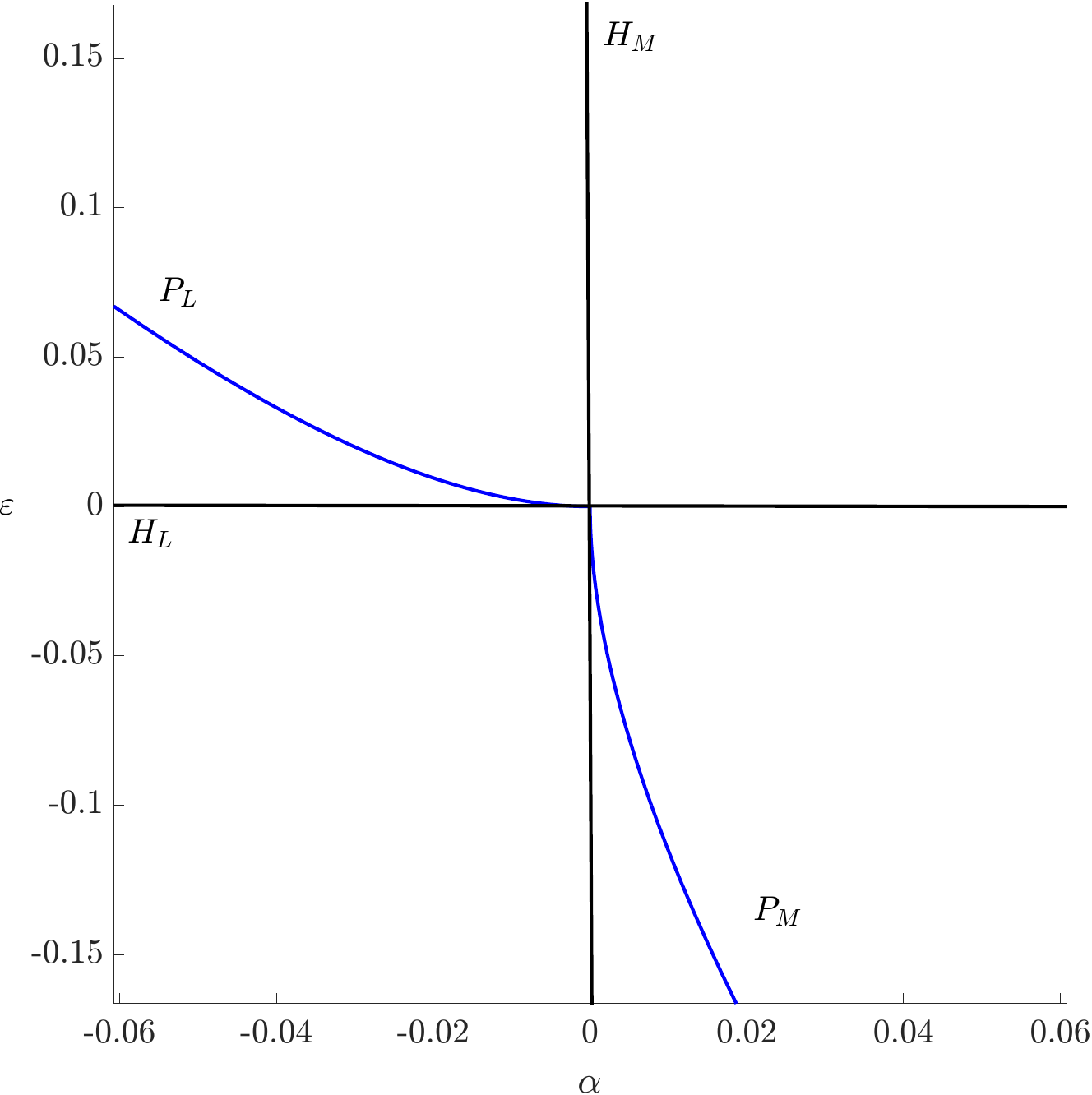}
    \label{fig:numericbifurcationdiagramcase1}
\end{figure}

\begin{figure}[H]
    \centering
    \caption{The parameter plane diagram in the second subcase (\(a=1\), \(b=-3\) and \(c=1/2\))}
    \includegraphics[width=0.4\textwidth]{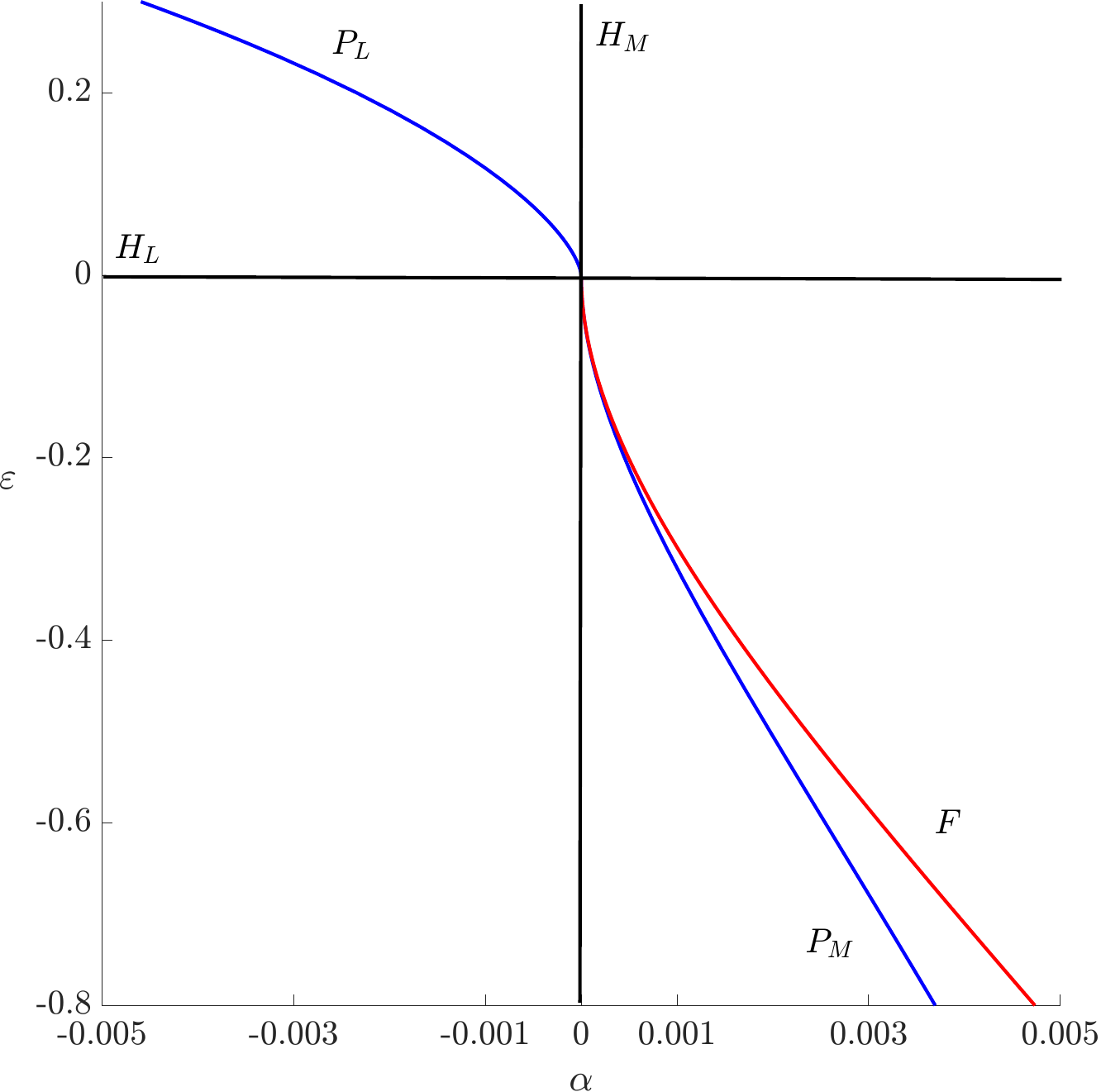}
    \label{fig:numericbifurcationdiagramcase2}
\end{figure}

\end{multicols}

\begin{figure}[H]
  \centering
  \includegraphics[width=\textwidth]{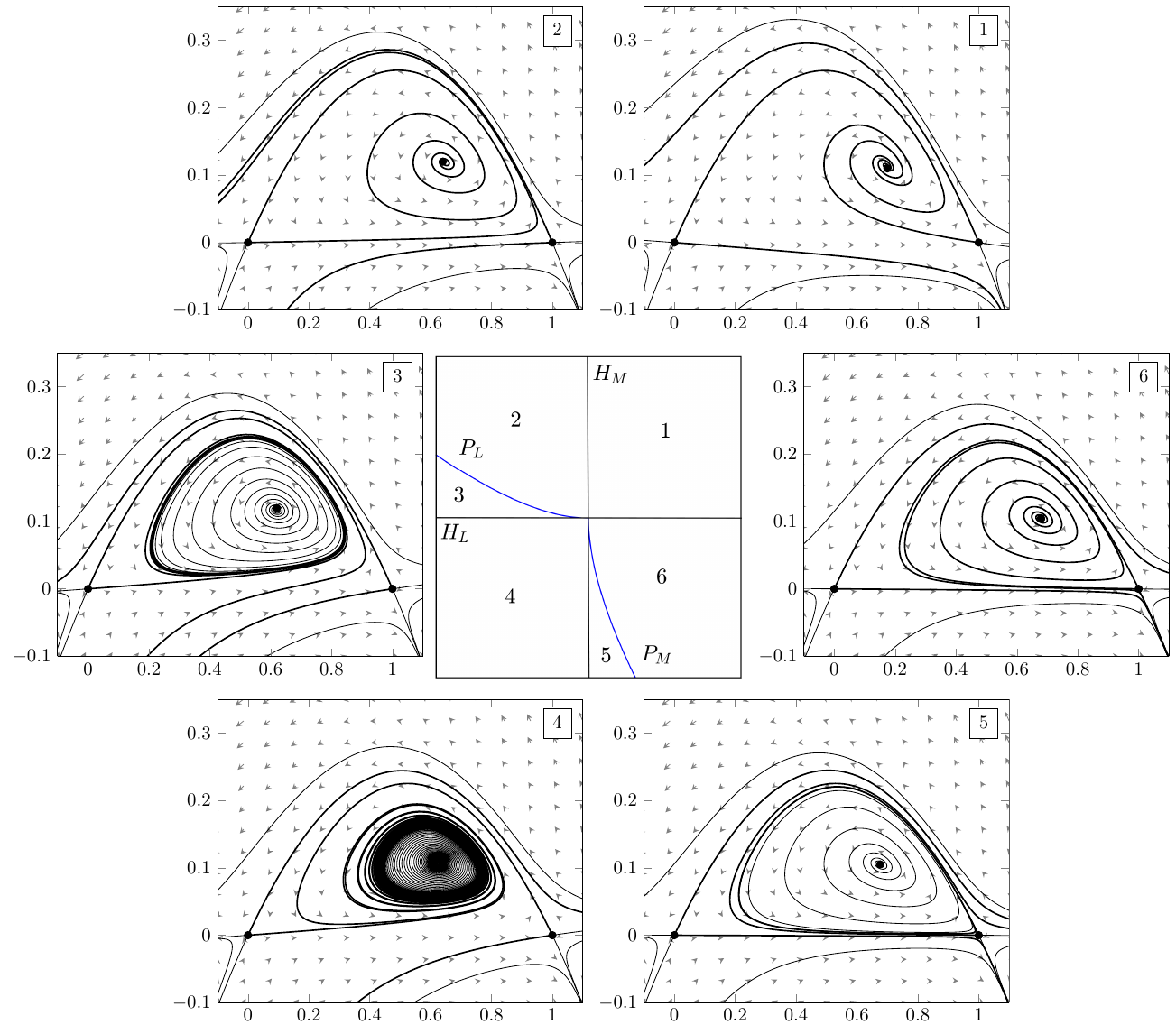}
  \caption{The full bifurcation diagram for the first subcase}
  \label{fig:totalhombifdiag1}
\end{figure}

\begin{figure}[H]
  \centering
  \includegraphics[width=\textwidth]{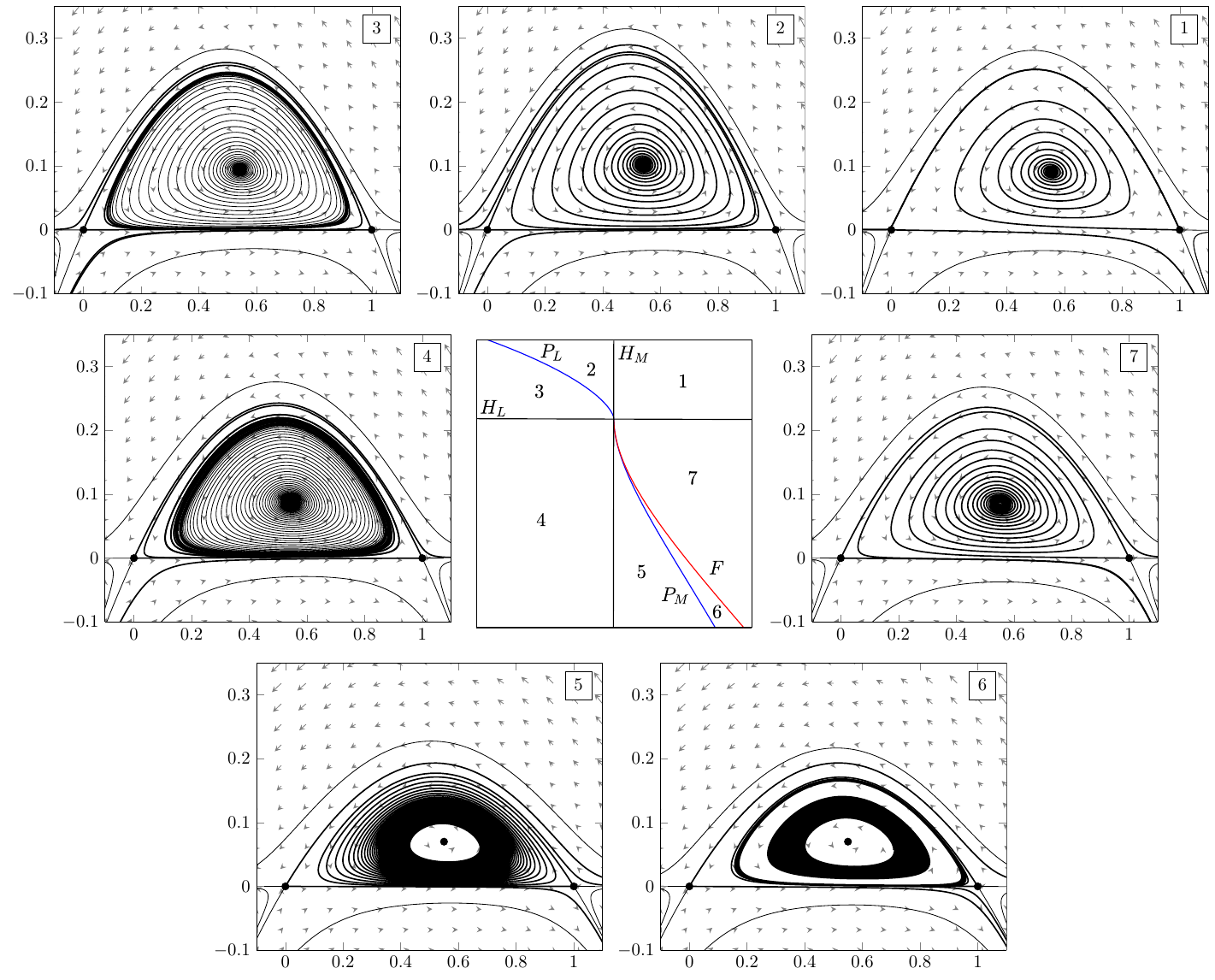}
  \caption{The full bifurcation diagram for the second subcase}
  \label{fig:totalhombifdiag2}
\end{figure}

\begin{multicols}{2}
\allowdisplaybreaks
\subsection{A non-monodromic example}

\subsubsection{Finding an example}
The starting point is the following reversible system\footnote{The system is invariant under the involution $(t,x,y) \mapsto (-t,-x,y)$.}
\begin{equation}
    \label{eq:Revers0}
    \left\{\begin{array}{rcl}
    \dot{x}&=&y,\\
    \dot{y}&=&x+xy-x^3,
    \end{array}
    \right.
\end{equation}
that has a saddle $(0,0)$, a stable focus $(-1,0)$, and an unstable focus $(1,0)$. The saddle at the origin has a `big' homoclinic orbit, while both foci are located inside the domain bounded by the homoclinic loop as seen in Fig. \ref{fig:Revers0}.
\begin{figure}[H]
    \centering
    \includegraphics{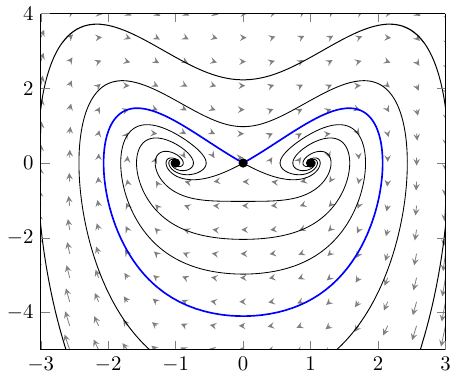}
    \caption{System (\ref{eq:Revers0}) containing a big homoclinic loop}
    \label{fig:Revers0}
\end{figure}

\begin{figure}[H]
    \begin{center}
    \includegraphics[]{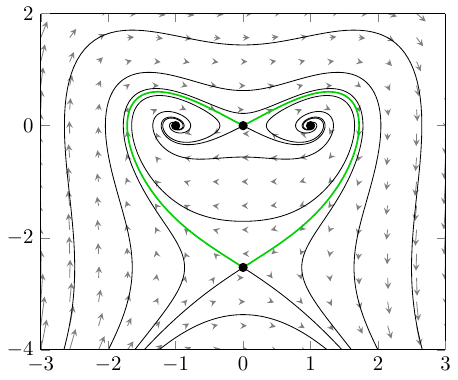}
    \end{center}
    \caption{The symmetric non-monodromic heteroclinic contour in the reversible system (\ref{eq:Revers1}) with $\gamma=\gamma_0$.}
    \label{fig:RevH}
\end{figure}
\begin{figure}[H]
    \begin{center}
    \includegraphics[width=0.45\textwidth]{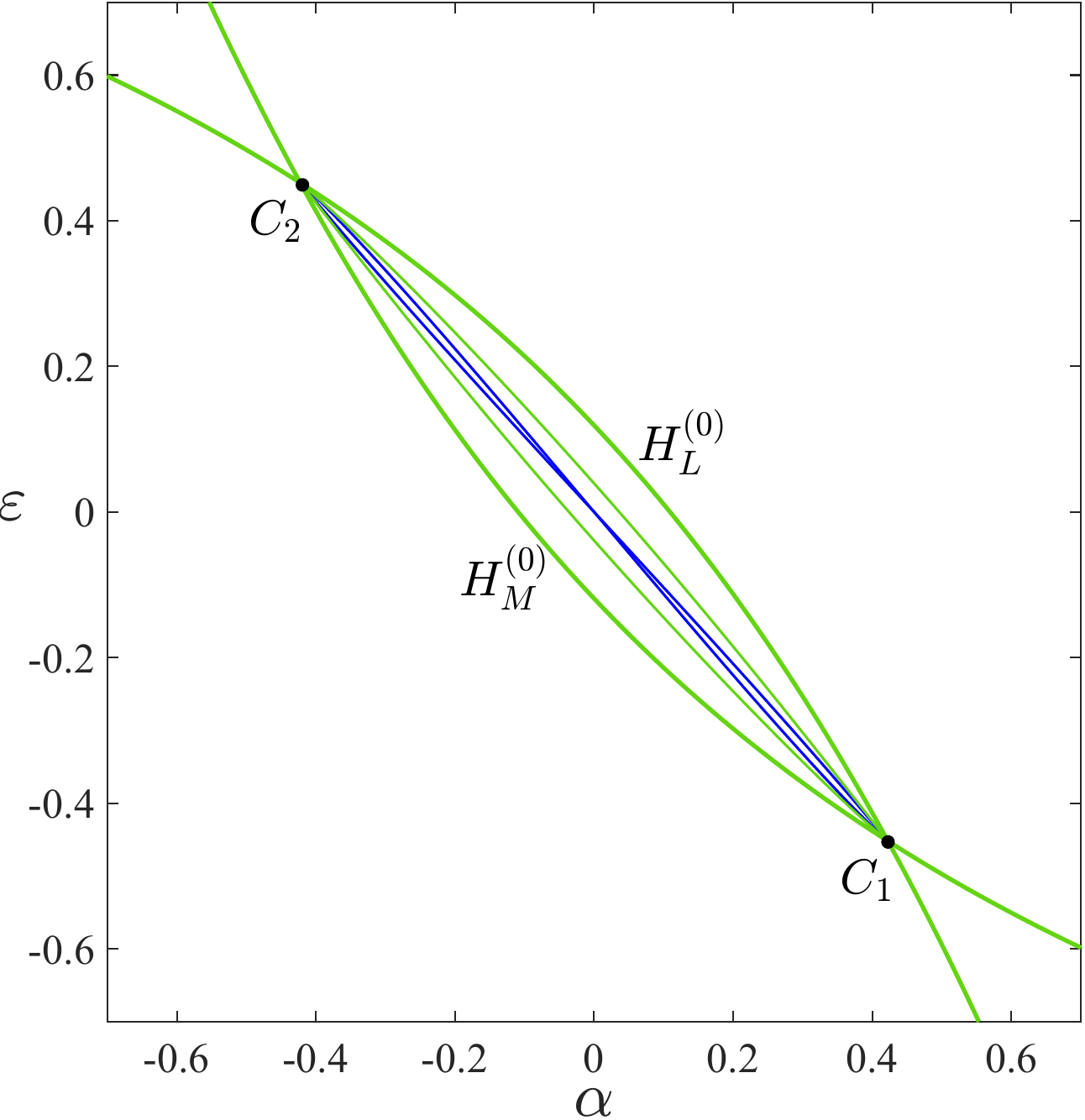}
    \end{center}
    \caption{Partial bifurcation diagram of (\ref{eq:DissH}) with $\gamma=2.7$: heteroclinic bifurcation curves (green), primary homoclinic curves (blue).}
    \label{fig:BifDiag}
\end{figure}
The one-parameter modification of (\ref{eq:Revers0}) that preserves the reversibility,
\begin{equation}
    \label{eq:Revers1}
    \left\{\begin{array}{rcl}
    \dot{x}&=&y(y+\gamma),\\
    \dot{y}&=&x+xy-x^3,
    \end{array}
    \right.
\end{equation}
has an additional saddle at $(0,-\gamma)$. This saddle is connected to the saddle at $(0,0)$ by two heteroclinic orbits when $\gamma_0=2.5315\ldots$ (see Fig. \ref{fig:RevH}). Of course,  both saddles of (\ref{eq:Revers1}) are neutral.

To break reversibility, introduce the following two-parameter deformation of (\ref{eq:Revers1}):
\begin{equation}
    \label{eq:DissH}
    \left\{\begin{array}{rcl}
    \dot{x}&=&\varepsilon x + y(y+\gamma),\\
    \dot{y}&=&\alpha y + x+xy-x^3.
    \end{array}
    \right.
\end{equation}
This system still has a saddle at $(0,0)$, while the second saddle moves off the $y$-axis. In order to obtain a generic non-monodromic heteroclinic contour, all three parameters of (\ref{eq:DissH}) have to vary. Since the existence of such heteroclinic contour is a codim 2 phenomenon, we expect a curve in the $(\alpha,\varepsilon,\gamma)$-space passing through $(0,0,\gamma_0)$, along which the contour persists.

\subsubsection{Numerical analysis of the example}
The main curves of the bifurcation diagram of (\ref{eq:DissH}) in the $(\alpha,\varepsilon)$-plane at $\gamma=2.7$ is shown in Fig. \ref{fig:BifDiag}. Since the transformation $(t,x,\alpha,\varepsilon) \mapsto (-t,-x,-\alpha,-\varepsilon)$ preserves the system (\ref{eq:DissH}), its bifurcation diagram is invariant under the inversion $(\alpha,\varepsilon) \mapsto (-\alpha,-\varepsilon)$.

\begin{figure}[H]
    \begin{center}
    \includegraphics[width=0.45\textwidth]{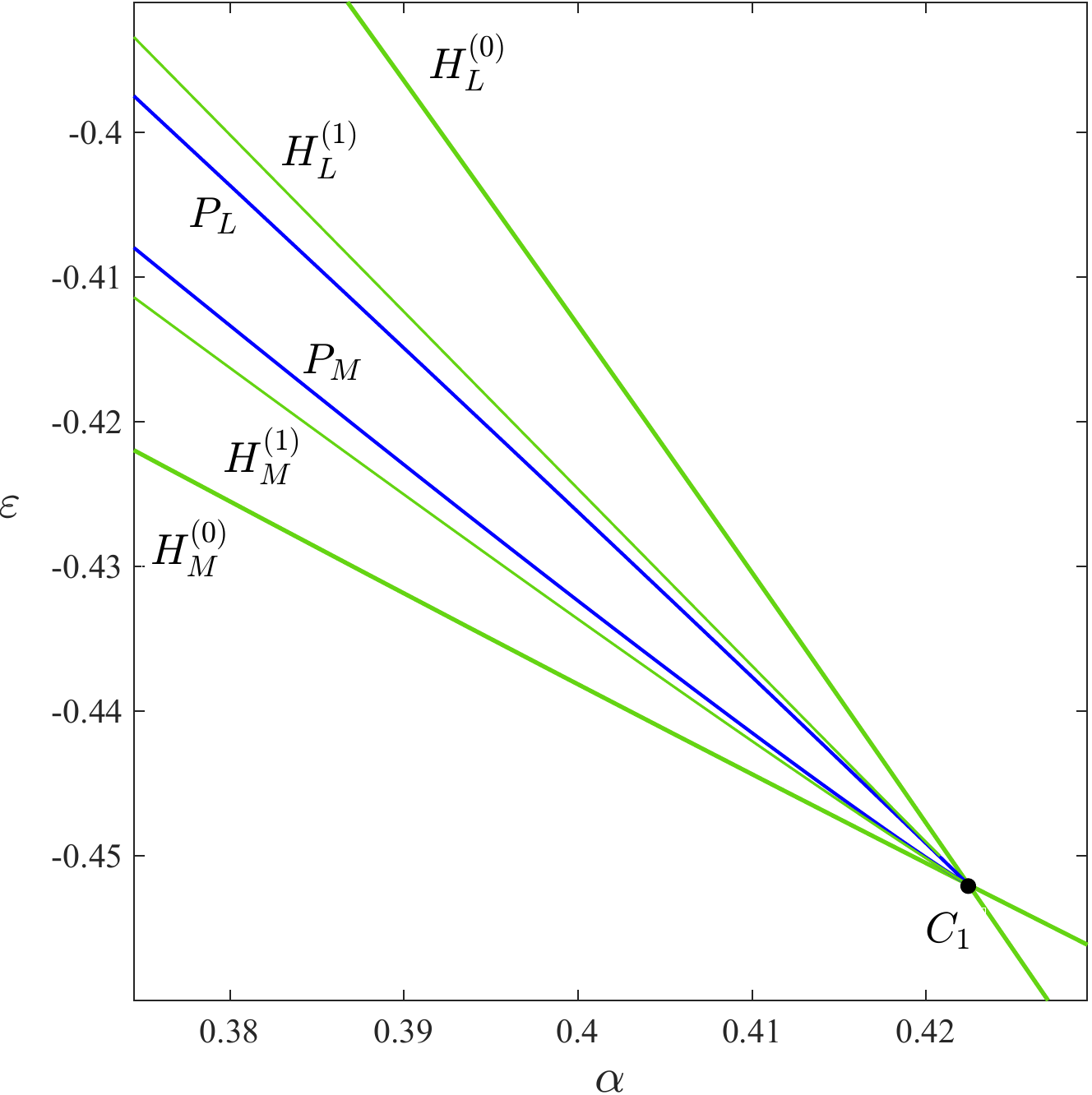}
    \end{center}
    \caption{Partial bifurcation diagram of (\ref{eq:DissH}) near the lower codim 2 point $C_1$: primary and secondary heteroclinic bifurcation curves $H_{L,M}^{(k)}$ (green), primary homoclinic curves $P_{L,M}$(blue).}
    \label{fig:ZoomBifDiag}
\end{figure}
There are two non-monodromic heteroclinic contours at the intersection points $C_{1,2}$ of the he\-te\-ro\-clinic curves $H_{L,M}^{(0)}$. These points are the end-points of two primary homoclinic curves $P_{L,M}$, which also intersect at $(\alpha,\varepsilon)=(0,0)$, where two independent homoclinic orbits to both saddles exist. Since the system is reversible at $\alpha=\varepsilon=0$, both homoclinic orbits are symmetric. All orbits in the annulus bounded by these two homoclinic orbits are periodic.

\par
The lower heteroclinic contour point is found to be
$$
C_1=(\alpha,\varepsilon)=(0.422432\ldots,-0.452007\ldots).
$$
The corresponding phase portrait is shown in Fig. \ref{fig:C1}. The saddles at these parameter values are
$$
L=(0,0)~~~{\rm and}~~~ M=(-0.57039\ldots,-2.6009\ldots),
$$
with eigenvalues
$$
\lambda_s=-1.7151\ldots,\ \ \lambda_u=1.6856\ldots
$$
and
$$
\mu_s=-2.8436\ldots,\ \ \mu_u=2.2436\ldots,
$$
respectively. The corresponding indices are
$$
\lambda=-\frac{\lambda_s}{\lambda_u}=1.0175\ldots>1
$$
and
$$
\mu=-\frac{\mu_s}{\mu_u}=1.2674\ldots>1.
$$
Notice that both saddles in this heteroclinic contour are dissipative.
\begin{figure}[H]
    \begin{center}
    \includegraphics[]{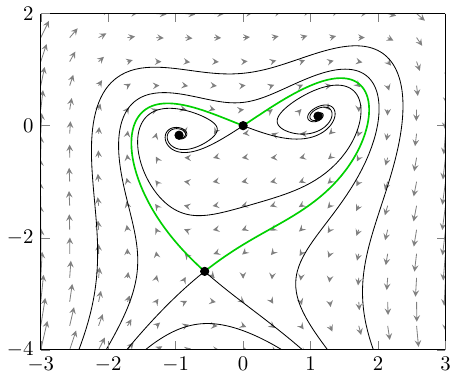}
    \end{center}
    \caption{Non-monodromic heteroclinic contour in system (\ref{eq:Revers1}) with $\gamma=2.7$ at codim 2 point $C_1$.}
    \label{fig:C1}
\end{figure}
\par
The upper heteroclinic contour point
$$
C_2=(\alpha,\varepsilon)=(-0.422432\ldots,0.452007\ldots).
$$
The corresponding phase portrait is shown in Fig. \ref{fig:C2}. The saddles at these parameter values are
$$
L=(0,0)~~~{\rm and}~~~ M=(0.570399\ldots,-2.6009\ldots),
$$
with the eigenvalues
$$
\lambda_s=-1.6856\ldots,\ \ \lambda_u=1.7151\ldots
$$
and
$$
\mu_s=-2.2436\ldots,\ \ \mu_u=2.8436\ldots,
$$
respectively. The corresponding indices are
$$
\lambda=-\frac{\lambda_s}{\lambda_u}=0.982\ldots < 1
$$
\rm and
$$
\mu=-\frac{\mu_s}{\mu_u}=0.789\ldots<1.
$$
\begin{figure}[H]
    \begin{center}
    \includegraphics[]{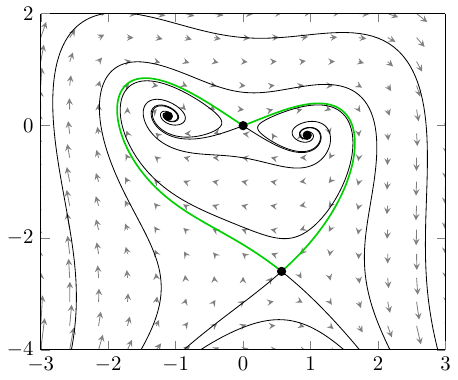}
    \end{center}
    \caption{Non-monodromic heteroclinic contour in system (\ref{eq:DissH}) with $\gamma=2.7$ at codim 2 point $C_2$.}
    \label{fig:C2}
\end{figure}
\par
\end{multicols}

\begin{figure}[H]
    \centering
    \includegraphics[width=0.9\textwidth]{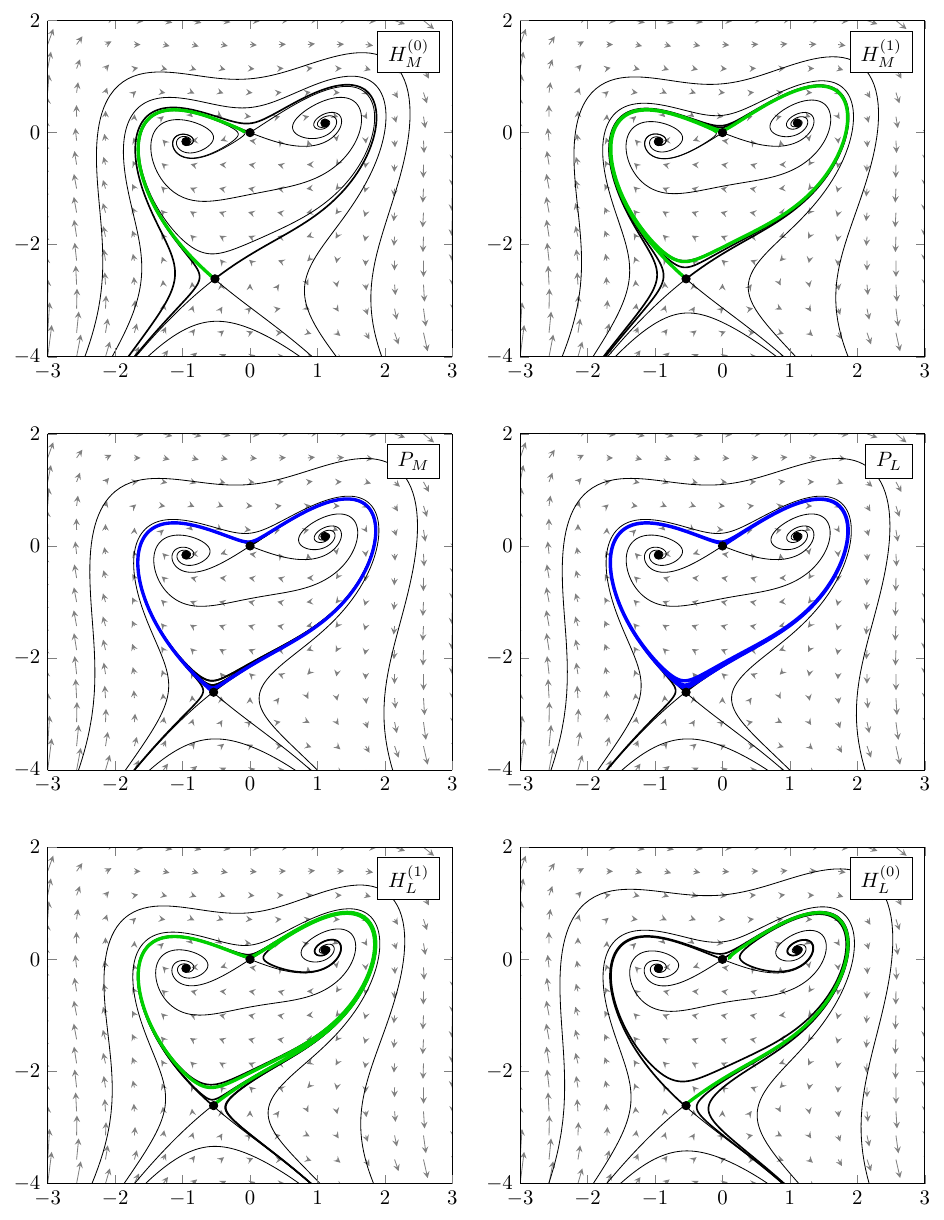}
    \caption{Phase portraits for critical values of the parameters}
    \label{fig:nonmomocritical}
\end{figure}
\vfill

\begin{multicols}{2}
\section{Discussion}

The theoretical results collected in this paper can be derived from the study of the truncated model maps. However, it must be shown that adding higher order terms to the truncated model maps does not change the results in generic cases, so that the leading terms are indeed sufficient to obtain all relevant theoretical results. The papers \cite{reyn,Roi:1989,Sha:1992,dukov}, on which our presentation is based, address this issue using different techniques.

There are two approaches leading to the same truncated Poincar\'{e} maps that differ by construction of singular correspondence maps near saddles. The first approach is based on fact that there is a local invertible $C^1$ coordinate transformation that maps solutions of a nonlinear planar system near a hyperbolic saddle onto the solutions of its linearization, see \cite{Deng:1989} or \cite{Ku:2004} for a more geometric treatment. This allows to compute the singular map exactly as $\xi \mapsto \xi^{\lambda}$, where $\xi$ is a coordinate in the unstable direction and $\lambda$ is the saddle index. Since for the analysis of cycle bifurcations further smoothness of the Poincar\'{e} map is required, one needs to impose more genericity conditions. Actually, there is a $C^k$ linearizing transformation with any $k>1$, if a finite number $N(k)$ of lower-order resonance conditions on the eigenvalues are not satisfied \cite{St:1958}. Thus the singular map is generically as smooth as required for $\xi>0$, provided the system (\ref{eq:planarsystem}) is sufficiently smooth.

There is an alternative approach (see, e.g. \cite{ShilnikovTuraev:2001}) that leads to the same truncated Poincar\'{e} maps but avoids a reduction to a linear system near the saddles and, thus, does not impose extra conditions. In this approach, one first introduces new coordinates near the saddle, such that the system remains nonlinear but the stable and unstable manifolds locally coincide with the axes (and some other simplifications are done). Then by considering a boundary-value problem, one proves that the singular map near the saddle has the form
$\xi \mapsto \xi^{\lambda}+o(\xi^{\lambda})$.
Composing such maps with the regular global maps, gives the same truncated Poincar\'{e} maps as in the first approach. This construction leads to a small finite loss of smoothness relative to that of the original planar system (\ref{eq:planarsystem}). The papers \cite{reyn,dukov} use the first approach (linearization), while the second one is employed in \cite{Roi:1989} and seems to be less restrictive. However, the linearization method might be preferable didactically.

In order to complete our collection of examples demonstrating the theoretical results of this thesis, one more example needs to be added. Of the non-monodromic contours, we were only able to produce an example of the simpler subcase (with \(\lambda<1\) and \(\mu<1\)). Addition of a non-monodromic example of the other subcase (with \(\lambda<1\) but \(\mu>1\)), displaying a fold of cycles bifurcation will complete the stock of examples.

The bifurcations discussed in this paper all have a semi-local nature. That is, they appear in a neighborhood of a heteroclinic contour. It is often the case that bifurcations which have a global or semi-local nature in systems with few parameters appear locally in systems with more parameters. For example, in \cite{BazKuzKhib:1989,Dum:1991} the monodromic case is shown to appear close to one case of a degenerate Bogdanov-Takens bifurcation of codim 3. For the non-monodromic case, an analogous result remains to be found. We expect that it must have codimension bigger than 3 and that the linear part of the associated critical normal form is identically zero.

\appendix{Flashing heteroclinic
connections}
\label{app:flashing}
The most exotic bifurcation phenomenon encountered in the paper is the appearance of flashing (or sparkling) heteroclinic connections. In fact, its appearance is one of the most important reasons to study bifurcations of non-monodromic heteroclinic contours, as it results in a bifurcation diagram which is surprisingly complex for a system which is only planar. Here we summarize for the reader's convenience the major facts on these bifurcations, closely following \cite{dukov}.

As noted in the introduction, this bifurcation was discovered by \citeauthor{maltapalis} [\citeyear{maltapalis}], see also \cite{ShilnikovTuraev:2001}. The flashing heteroclinic connections refer to the generation of an infinite series of heteroclinic connections between two saddles which make an increasing number of turns around one of the saddles before reaching the other.

The situation in which this bifurcation was first discovered is associated with a non-hyperbolic (double) limit cycle. Of the two saddles involved in this bifurcation, one is located inside the limit cycle while the other is outside\footnote{Of course, more equilibria must be encircled by the cycle.}. The non-hyperbolic limit cycle serves as the $\omega$-limit set of an orbit leaving the outside saddle, while simultaneously it is the $\alpha$-limit set for an orbit approaching the other saddle for increasing \(t\). Thus, there are orbits winding around the cycle from both the outside and inside.

The following theorem shows that when this limit cycle disappears, an infinite series of heteroclinic connections is generated which wind around the inside saddle an increasing number of times.

In the theorem we consider the parameter \(\alpha\in\mathbb{R}\) involved in the fold of cycles bifurcation. Thus, for \(\alpha=0\) there exists a semi-stable limit cycle as described above and for \(\alpha>0\) this cycle disappears.

\begin{theorem}\label{thm:flashingcycle}
    Consider a vector field containing two saddles and a semi-stable limit cycle as described above. Then there exists a sequence \((\alpha_n)_{n\in\mathbb{N}}\) with \(\alpha_n>0\) for all \(n \in \mathbb{N}\) and \(\alpha_n\to0\) such that for every \(n\in \mathbb{N}\), the vector field corresponding to the parameter value \(\alpha_n\) contains a heteroclinic connection making more turns around the inner saddle then the previous one.
\end{theorem}

\begin{proof}
    For \(\alpha=0\), the vector field contains a semi-stable limit cycle which separates the saddles \(L\) and \(M\). On this cycle we select a cross-section \(\Sigma\) on which we define a coordinate \(x\) directed outwards (see Fig. \ref{fig:cycledetails}). Suppose that the limit cycle and the cross-section intersect at \(x=0\). The orbit approaching the saddle \(L\) departs from the cross-section at \(x_1(0)<0\) and the orbit leaving from \(M\) arrives to the cross-section at \(x_2(0)>0\).

    \begin{figure}[H]
        \centering
        \includegraphics{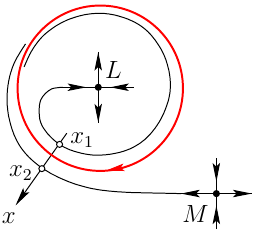}
        \caption{Placing a cross-section on the cycle}
        \label{fig:cycledetails}
    \end{figure}

    When we pick a small \(\alpha>0\) and study the vector field still furnished with the cross-section \(\Sigma\), we see that we get parameter dependent \(x_1(\alpha)<0<x_2(\alpha)\) and that the semi-stable limit cycle is destroyed, so that the parameter-dependent Poincaré map \(P_{\alpha}\) has no fixed points. Therefore, we have that \(P_\alpha(x)<x\) for all \(x \in (x_1(\alpha),x_2(\alpha)]\) and \(P^{-1}_\alpha(x)>x\) for all \(x \in [x_1(\alpha),x_2(\alpha))\). Thus, after enough iterations of the Poincaré map and its inverse, we will find \(n\in \mathbb{N}\) such that  \(P_\alpha^{(n)}(x_2(\alpha))<0\) and \(P_\alpha^{(-n)}(x_1(\alpha))>0\).

    In contrast to this, we always have that \(P^{(n)}_0(x_2(0))>0\) and \(P^{(-n)}_0(x_1(0))<0\). Thus, for any sufficiently big $n$, we will find \(\alpha_n\) such that
    \(P_{\alpha_n}^{(n)}(x_2(\alpha_n))=P_{\alpha_n}^{(-n)}(x_1(\alpha_n))\) implying a heteroclinic connection with more and more turns (see Fig. \ref{fig:flashing}).
\end{proof}

\begin{figure}[H]
  \centering
  \subfloat[No turns]{\includegraphics[width=0.2\textwidth]{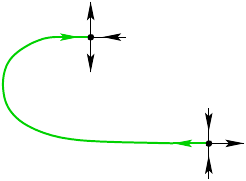}}
  \subfloat[One turn]{\includegraphics[width=0.2\textwidth]{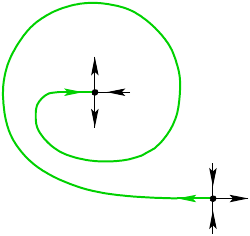}}\\
  \subfloat[Two turns]{\includegraphics[width=0.2\textwidth]{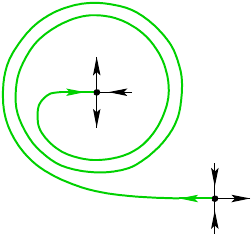}}
  \caption{Heteroclinic connections making an increasing number of turns}
  \label{fig:flashing}
\end{figure}

The occurrence of flashing heteroclinic connections is not limited to situations involving a cyclic fold bifurcation. They may also appear in a similar situation involving the bifurcation of a homoclinic connection. More specifically, we consider the situation in which there exist two saddles and a homoclinic connection at one of the saddles. The other saddle is located inside this saddle. We consider a system depending on the splitting parameter \(\beta\) of the homoclinic connection. Suppose the saddle at which the homoclinic connection exists has saddle index smaller than one, so for \(\beta>0\) the homoclinic connection is broken and no limit cycle is generated.

\begin{theorem}\label{thm:flashingloop}
  Consider a vector field containing two saddles and a homoclinic connection as described above. Then there exists a sequence \((\beta_n)_{n\in\mathbb{N}}\) with \(\beta_n>0\) for all \(n \in \mathbb{N}\) and \(\beta_n\to0\) such that for every \(n\in \mathbb{N}\), the vector field corresponding to the parameter values
  \(\beta_n\) contains a heteroclinic connection making more turns around the inner saddle then the previous one.
\end{theorem}

\begin{proof}
    The vector field for \(\beta=0\) contains a homoclinic connection at \(M\) and \(L\) lies on the inside of this connection. The orbit leaving \(L\) has the homoclinic connection as its limit set. On this connection we consider a cross-section \(\Sigma\) on which we define a coordinate \(x\) directed inwards (see Fig. \ref{fig:loopdetails}). Suppose that the homoclinic connection (that is, the orbit leaving and approaching \(M\)) and its cross-section intersect at \(x=0\) and that the solution leaving the saddle \(L\) reaches the cross-section at \(x_1(0)>0\).

    \begin{figure}[H]
        \centering
        \includegraphics{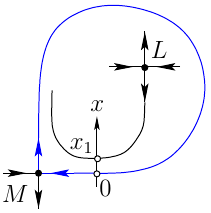}
        \caption{Placing a cross-section on the homoclinic connection}
        \label{fig:loopdetails}
    \end{figure}

    When we pick a small \(\beta>0\) and study the vector field, still furnished with the cross-section \(\Sigma\), we see that we get parameter dependent \(x_1(\beta)>0\) and that the homoclinic connection is destroyed. Moreover, the parameter dependent Poincaré map \(P_{\beta}\) has no fixed points. Therefore, we have that \(P_\beta(x)<x\) so after enough iterations of the Poincaré map we will find \(n\in \mathbb{N}\) such that \(P_\beta^{(n)}(x_1(\beta))<0\).

    In contrast to this, we always have that \(P^{(n)}_0(x_1(0))>0\). Thus, for all sufficiently big $n$, we will find \(\beta_n\) such that \(P_{\beta_n}^{(n)}(x_1(\beta_n))=0\) which corresponds to a heteroclinic connection with more and more turns (see Fig. \ref{fig:flashingloops}).
\end{proof}

\begin{figure}[H]
  \centering
  \subfloat[No turns]{\includegraphics[width=0.2\textwidth]{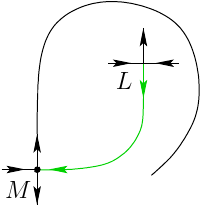}}
  \subfloat[One turn]{\includegraphics[width=0.2\textwidth]{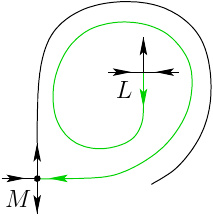}}\\
  \subfloat[Two turns]{\includegraphics[width=0.2\textwidth]{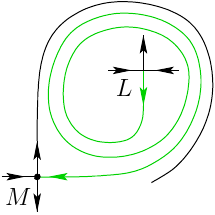}}
  \caption{Heteroclinic connections making an increasing number of turns}
  \label{fig:flashingloops}
\end{figure}

\appendix{Regularity of splitting}\label{app:melnikov}
In this appendix we show that the splitting of the heteroclinic connections in the planar system given by (\ref{eq:perturbedmonodromicvectorfield}) is regular by verifying that the relevant Melnikov integrals are nonzero when evaluated along the heteroclinic connections. An introduction to this technique can be found, e.g., in \cite{Ku:2004}.

Writing a smooth planar system dependent on a parameter $p \in\mathbb{R}$ as
\begin{equation}
\left\{\begin{array}{rcl}
\dot{x}&=&f(x,y,p),\\
\dot{y}&=&g(x,y,p),
\end{array}
\right.
\label{eq:2DappendixB}
\end{equation}
the Melnikov integral becomes
\[M_p(0)=\int_{-\infty}^\infty\varphi(t)\psi(t)dt\, ,\]
where \[\varphi(t)=\exp\left[-\int_0^t\left(\frac{\partial f}{\partial x}+\frac{\partial g}{\partial y}\right)d\tau\right]\] and \[\psi(t)=f\frac{\partial f}{\partial p}-g\frac{\partial f}{\partial p}\] have to be evaluated along the heteroclinic solutions of the system at the critical parameter value $p=0$.

The heteroclinic connection in (\ref{eq:perturbedmonodromicvectorfield}) along \(y=0\) splits under variation of the parameter \(\alpha\). Setting \(p=\alpha\) and \(y=0\) we find that \[\psi(t)=-x(t)^2(1-x(t))^2<0\] for all \(t\in\mathbb{R}\). As \(\varphi(t)>0\) for all \(t\) as well, we conclude that \(M_\alpha(0)\neq0\) independent on the parameters \(a\), \(b\) and \(c\) of (\ref{eq:perturbedmonodromicvectorfield}).

The connection along \(y=x(1-x)\) splits under variation of the parameter \(\varepsilon\) so we set \(p=\varepsilon\) and \(y=x(1-x)\). In this case, we transform the integrals with respect to $t$ to integrals in \(x\) over the domain \([0,1]\), and proceed numerically. We explicitly work out the case \(c=1/2\). The remaining case is similar. Using the first equation in (\ref{eq:perturbedmonodromicvectorfield}), we obtain that \[\psi(t(x))\,dx=x(-2x^2+3x-1)\,dx.\] Similarly, the integrand involved in \(\varphi(t(x))\) becomes \[\frac{5x^2-8x+2}{x^3-5x^2+4x}\,dx.\] Choosing the particular solution satisfying \(x(0)=1/2\) we obtain \[\varphi(t(x))=\exp\left[-\int_{1/2}^x\frac{5\xi^2-8\xi+2}{\xi^3-5\xi^2+4\xi}\,d\xi\right].\] As \(\lim_{t\to\infty}x(t)=0\), the total integral becomes \[M_\varepsilon(0)=-\int_0^1\varphi(t(x))\psi(t(x))dx\]
Numerical evaluation yields \(M_\varepsilon(0)=-0.0177888\). Analogously, we find that \(M_\varepsilon(0)=-0.0866532\) when \(c=3/2\), showing that the splitting happens regularly in both cases.
\par\bigskip
\end{multicols}

\begin{multicols}{2}

\bibliographystyle{ws-ijbc}
\bibliography{bibl}

\end{multicols}

\end{document}